\documentclass[preprint,review,12pt]{elsarticle}

\usepackage{amssymb}
\usepackage[leqno]{amsmath}
\usepackage[german,english]{babel}
\usepackage{graphicx}
\usepackage{subfigure}
\usepackage{cases}
\usepackage{amsfonts}
\usepackage{amsthm}
\usepackage{booktabs}
\usepackage{hyperref}
\usepackage{epstopdf}
\usepackage{color}
\usepackage{float}

\usepackage{mathrsfs}

\numberwithin{equation}{section}

\makeatletter
\def\fnum@figure{Fig.\thefigure}
\makeatother

\theoremstyle{definition}
\newtheorem*{assumption*}{Assumption}

\newtheorem{remark}{Remark}
\numberwithin{remark}{section}
\newtheorem*{remark*}{Remark}

\newtheorem{lemma}{Lemma}
\numberwithin{lemma}{section}
\newtheorem{theorem}{Theorem}
\numberwithin{theorem}{section}

\let \mr=\mathrm

\begin{document}

\begin{frontmatter}




\title{Analysis of SDFEM on Shishkin triangular meshes and hybrid meshes for problems with characteristic layers\tnoteref{label3}}
\tnotetext[label3]{This research was partly supported by NSF of China  (
Grant Nos. 11501335,  11401349 and 11501334) and NSF of Shandong Province (Grant Nos. BS2014SF008 and ZR2015FQ014).}

\author[label1] {Jin Zhang\corref{cor1}}
\author[label2] {Xiaowei Liu\fnref {cor2}}
\cortext[cor1] {Corresponding author:   jinzhangalex@hotmail.com }
\fntext[cor2] {Email: xwliuvivi@hotmail.com }
\address[label1]{School of Mathematical Sciences, Shandong Normal University,
Jinan 250014, China}
\address[label2]{College of Science, Qilu University of Technology, Jinan 250353, China.}

\begin{abstract}
In this paper,  we analyze the  streamline diffusion finite element method (SDFEM) for  a model singularly perturbed convection-diffusion equation on  a Shishkin triangular mesh and  hybrid meshes.  
Supercloseness property of $u^I-u^N$ is obtained, where  $u^I$ is  the interpolant of the solution $u$ and  $u^N$ is the SDFEM's solution.
The analysis depends on novel integral inequalities for the diffusion and convection parts in the bilinear form.
Furthermore, analysis on    hybrid meshes shows that bilinear elements should be recommended for  the exponential layer,  not  for the characteristic layer.  Finally, numerical experiments support these theoretical results.

\end{abstract}

\end{frontmatter}
%
%
%

%
%
\section{Introduction}
We consider the singularly perturbed boundary value problem
 \begin{equation}\label{eq:model problem}
 \begin{array}{rcl}
-\varepsilon\Delta u+b  u_x+cu=f & \mbox{in}& \Omega=(0,1)^{2},\\
 u=0 & \mbox{on}& \partial\Omega ,
 \end{array}
 \end{equation}
where $\varepsilon\ll |b|$ is a small positive parameter, the functions $b(x,y)$, $c(x,y)$ and $f(x,y)$ are supposed sufficiently smooth. We also assume  $$b(x,y)\ge \beta>0, \, c(x,y)-\frac{1}{2}b_x(x,y)\ge \mu_{0}>0\quad \text{on $\bar{\Omega},$} $$
where $\beta$ and $\mu_{0}$ are some constants. The solution of \eqref{eq:model problem} typically has
an exponential layer of width
$O(\varepsilon\ln(1/\varepsilon) )$ near the outflow boundary at $x=1$ and two characteristic (or parabolic) layers of width $O(\sqrt{\varepsilon}\ln(1/\varepsilon))$ near the characteristic boundaries at $y=0$
and $y=1$.
\par
Because of the presence of layers,  standard finite element methods suffer from nonphysical oscillations 
unless meshes are taken sufficiently fine which are useless for practical purposes. Thus, stabilized methods and/or a priori adapted meshes (see \cite{Stynes:2005-Steady, Roo1Sty2Tob3:2008-Robust}) are widely used in order to get discrete solutions with satisfactory
stability and accuracy.  Among them,  the streamline diffusion finite element method (SDFEM)  \cite{Hugh1Broo2:1979-multidimensional} combined with the Shishkin mesh \cite{Shishkin:1990-Grid} presents good numerical performances and has been widely studied, see \cite{Styn1Tobi2:2003-SDFEM,Fra1Lin2Roo3:2008-Superconvergence,Fra1Kel2Sty3:2012-Galerkin,Styn1Tobi2:2008-Using}.


In this work, we will analyze  supercloseness property of the SDFEM  for problem \eqref{eq:model problem}. Here ``supercloseness''  means the convergence order  of $u^I-u^{N}$ in some norm is greater than one  of $u-u^I$. This property  in the case of rectangular meshes has been analyzed in  \cite{Styn1Tobi2:2003-SDFEM,Fra1Lin2Roo3:2008-Superconvergence} by means of integral identities \cite{Lin1Yan2Zho3:1991-rectangle} and it is helpful to 
derive optimal $L^2$ estimates,   $L^{\infty}$ bounds and  postprocessing procedures. Unfortunately, on triangular meshes few results of supercloseness property could be found up to now. In this article, we present it in Theorem \ref{eq: uI-uN varepsilon} by means of novel integral inequalities, i.e., Lemmas \ref{lem:superconvergence-diffusion-term} and \ref{lem:superconvergence-convection-term}. 
Furthermore, the SDFEM is analyzed on Shishkin hybrid  meshes
which consist of rectangles and triangles.  Theorem \ref{theorem:superconvergence-Hybrid mesh-CL} shows that rectangles are strongly recommended for the exponential layer and  not necessary for the characteristic layer.

 Here is the outline of this article. In \S 2 we give some a priori information for the solution of \eqref{eq:model problem}, then introduce a Shishkin mesh   and a streamline diffusion finite element method on the  mesh. In \S 3 we present integral inequalities and the interpolation errors.
In \S 4 we analyze the supercloseness property on the Shishkin triangular mesh. In \S 5 we obtain supercloseness property again on hybrid meshes. Finally, some numerical results are presented in \S 6.

Throughout the article, the standard notations for the Sobolev spaces
and norms will be used; and generic constants $C$, $C_i$ are independent of
$\varepsilon$ and $N$.
An index will be attached to indicate an inner product or a norm on a subdomain $D$,
for example, $(\cdot, \cdot)_{D}$ and $\Vert \cdot \Vert_{D}$.

\section{Regularity results, Shishkin meshes and the SDFEM}
\subsection{Regularity results}
As mentioned before the solution $u$ of \eqref{eq:model problem} possesses an exponential layer at $x=1$ and two characteristic layers at $y=0$ and $y=1$. For our later analysis we shall  make the following assumption.
\newtheorem{assumption}[theorem]{Assumption}
\begin{assumption}\label{assumption-regularity}
The solution $u$ of \eqref{eq:model problem} can be decomposed as
\begin{subequations}
\begin{equation}\label{eq:(2.1a)}
u=S+E_{1}+E_{2}+E_{12},\quad \forall (x,y)\in\bar{\Omega}.
\end{equation}
For $0\le i+j \le 3$, the regular part satisfies
\begin{equation}\label{eq:(2.1b)}
\left|\frac{\partial^{i+j}S}{\partial x^{i}\partial y^{j}}(x,y) \right|\le C,
\end{equation}
while for  $0\le i+j \le 3$, the layer terms  satisfy
\begin{equation}\label{eq:(2.1c)}
\left|\frac{\partial^{i+j}E_{1}}{\partial x^{i}\partial y^{j}}(x,y) \right|\le C\varepsilon^{-i}e^{-\beta(1-x)/\varepsilon},
\end{equation}
\begin{equation}\label{eq:(2.1d)}
\left|\frac{\partial^{i+j}E_{2}}{\partial x^{i}\partial y^{j}}(x,y) \right|\le C\varepsilon^{-j/2}(e^{-y/\sqrt{\varepsilon}}+e^{-(1-y)/\sqrt{\varepsilon}}),
\end{equation}
and
\begin{equation}\label{eq:(2.1e)}
\quad\quad\quad
\left|\frac{\partial^{i+j}E_{12}}{\partial x^{i}\partial y^{j}}(x,y) \right|\le C\varepsilon^{-(i+j/2)}e^{-\beta(1-x)/\varepsilon}
(e^{-y/\sqrt{\varepsilon}}+e^{-(1-y)/\sqrt{\varepsilon}}).
\end{equation}
\end{subequations}
\end{assumption}
\begin{remark}
In \cite{Kell1Styn2:2005-Corner,Kell1Styn2:2007-Sharpened} Kellogg and Stynes presented sufficient compatibility conditions on $f$ for constant functions $b$, $c$ that ensure
the existence of \eqref{eq:(2.1a)}--\eqref{eq:(2.1e)}.
\end{remark}

\subsection{Shishkin meshes}
 When discretizing \eqref{eq:model problem}, first we divide the domain $\Omega$  into four(six) subdomains as $\bar{\Omega}=\Omega_{s}\cup\Omega_{x}\cup\Omega_{y}\cup\Omega_{xy}$(see Fig. \ref{fig:Shishkin mesh}), where
\begin{align*}
&\Omega_{s}:=\left[0,1-\lambda_{x}\right]\times\left[\lambda_{y},1-\lambda_{y}\right],&&
\Omega_{y}:=\left[0,1-\lambda_{x}\right]\times\left(\left[0,\lambda_{y}\right]
\cup\left[1-\lambda_{y},1\right]
\right),\\
&\Omega_{x}:=\left[ 1-\lambda_{x},1 \right]\times\left[\lambda_{y},1-\lambda_{y}\right],&&
\Omega_{xy}:=\left[ 1-\lambda_{x},1 \right]\times\left(\left[0,\lambda_{y}\right]
\cup\left[1-\lambda_{y},1\right]
\right).
\end{align*}
Two parameters $\lambda_x$ and $\lambda_y$ are  used here for mesh transition from coarse to fine and  are defined by
\begin{equation*}
\lambda_{x}:=\min\left\{ \frac{1}{2},\rho\frac{\varepsilon}{\beta}\ln N \right\} \quad \mbox{and} \quad
\lambda_{y}:=\min\left\{
\frac{1}{4},\rho\sqrt{\varepsilon}\ln N \right\}.
\end{equation*}
For technical reasons, we set $\rho=2.5$. Moreover, we  assume $\varepsilon\le \min\{ N^{-1}, \ln^{-6} N \}$ and 
$$
\lambda_{x}=\rho\varepsilon\beta^{-1}\ln N\le \frac{1}{2}\quad \text{and}\quad \lambda_{y}=\rho\sqrt{\varepsilon}\ln N\le 
\frac{1}{4}
$$
as is typically the case for \eqref{eq:model problem}.


\begin{figure}
\begin{minipage}[t]{0.5\linewidth}
\centering
\includegraphics[width=2.5in]{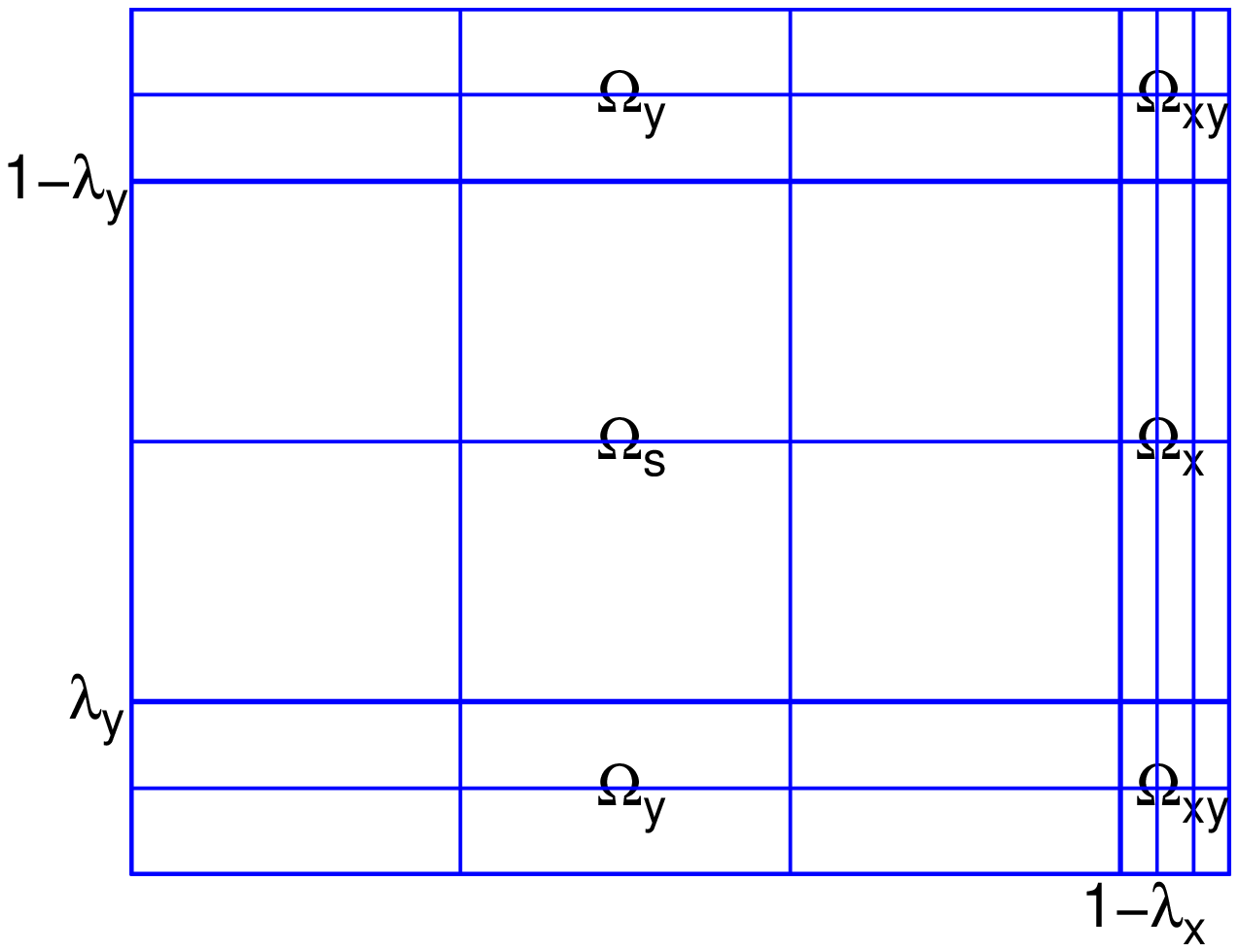}
\caption{Dissection of $\Omega$ and triangulation $\mathcal{T}_{N}$.}
\label{fig:Shishkin mesh}
\end{minipage}%
\begin{minipage}[t]{0.5\linewidth}
\centering
\includegraphics[width=2.5in]{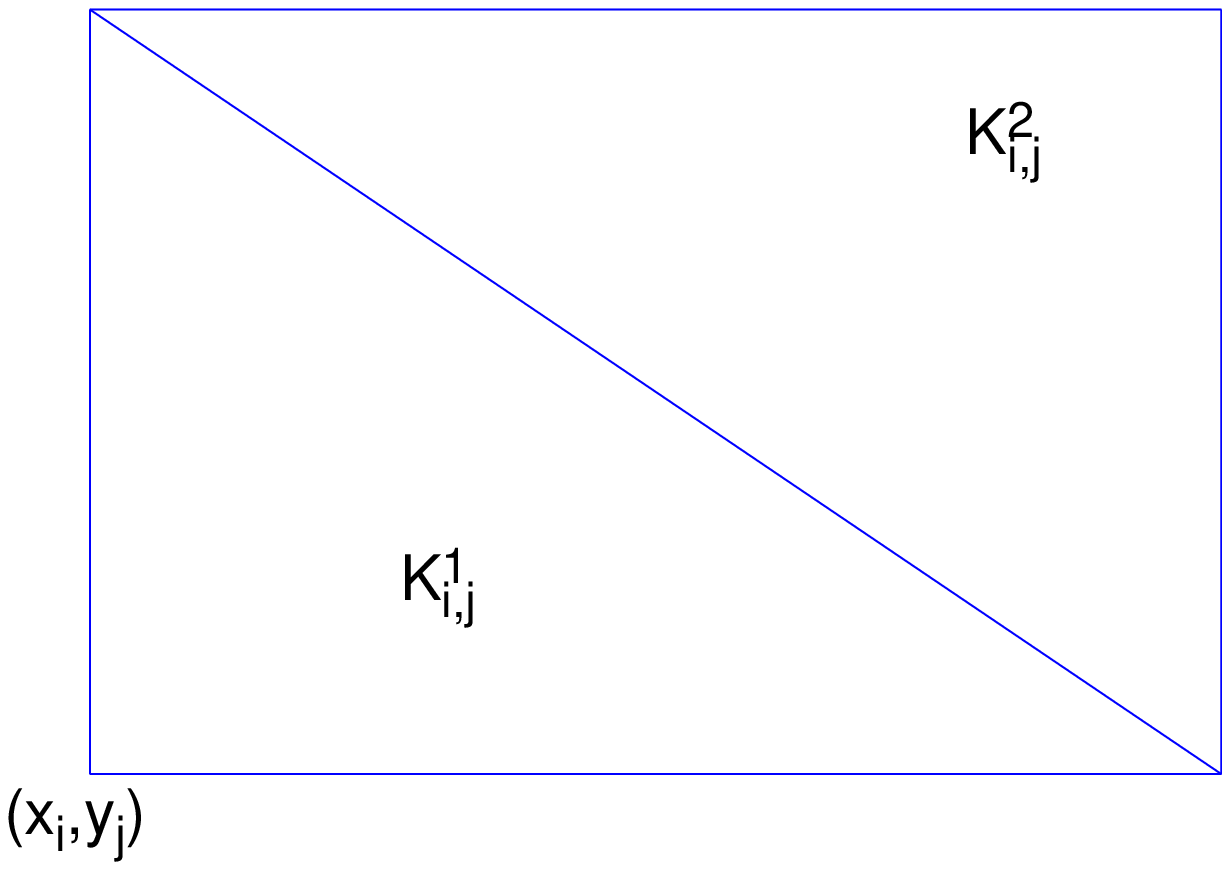}
\caption{$K^{1}_{i,j}$ and $K^{2}_{i,j}$}
\label{fig:code of mesh}
\end{minipage}
\end{figure}


\par
Next, we introduce the set of mesh points $\left\{ (x_{i},y_{j})\in\bar{\Omega}:\; i,\,j=0,\,\cdots,\,N  \right\}$ defined by
\begin{numcases}{x_{i}=}
2i(1-\lambda_{x})/N ,&\text{for $i=0,\,\cdots,\,N/2$}, \nonumber\\
1-2(N-i)\lambda_{x}/N, &\text{for $i=N/2+1,\,\cdots,\,N$}\nonumber
\end{numcases}
and
\begin{numcases}{y_{j}=}
3j\lambda_{y}/N ,&\text{for $j=0,\,\cdots,\,N/3$}, \nonumber\\
(3j/N-1)-3(2j-N)\lambda_{y}/N ,&\text{for $j=N/3+1,\,\cdots,\,2N/3$}, \nonumber\\
1-3(N-j)\lambda_{y}/N, &\text{for $j=2N/3+1,\,\cdots,\,N$}.\nonumber
\end{numcases}
By drawing lines through these mesh points parallel to the $x$-axis and $y$-axis,  the domain $\Omega$ is partitioned into rectangles  and triangles by drawing the diagonal 
in each rectangle (see Fig. \ref{fig:Shishkin mesh}). This yields a piecewise uniform triangulation of $\Omega$  denoted by $\mathcal{T}_{N}$.

We define $h_{x,i}:=x_{i+1}-x_{i}$ and $h_{y,j}:=y_{j+1}-y_{j}$  which
satisfy
\begin{align*}
N^{-1}\le &h_{x,i}=:H_x,h_{y,j}=:H_y \le 3N^{-1} ,\quad 0\le i < N/2,\; N/3 \le j <2N/3,\\
C_{1}\varepsilon N^{-1}\ln N &\le h_{x,i}=:h_x\le C_{2}\varepsilon N^{-1}\ln N,\quad N/2\le i <N,\\
C_{1}\sqrt{\varepsilon} N^{-1}\ln N &\le h_{y,j}=:h_y\le C_{2}\sqrt{\varepsilon} N^{-1}\ln N,\quad
j=0,\ldots,N/3-1;\; 2N/3,\ldots,N-1.
\end{align*}

 For mesh elements we shall use some notations: $K^1_{i,j}$ for the mesh triangle with vertices $(x_i,y_j)$, $(x_{i+1},y_j)$ and $(x_i,y_{j+1})$; $K^2_{i,j}$  for the mesh triangle with vertices
$(x_i,y_{j+1})$, $(x_{i+1},y_j)$ and $(x_{i+1},y_{j+1})$ (see Fig. \ref{fig:code of mesh});  $K$ for a generic mesh triangle.

\subsection{The streamline diffusion finite element method}
The variational formulation of
 problem \eqref{eq:model problem} is: 
\begin{equation}\label{eq:weak formulation}
\left\{
\begin{array}{lr}
\text{Find $u\in V$ such that for all $v\in V$}\\
\varepsilon (\nabla u,\nabla v)+(b u_{x}+cu,v)=(f,v),
\end{array}
\right.
\end{equation}
where $V:=H^{1}_{0}(\Omega)$. Note that the weak formulation \eqref{eq:weak formulation} has a unique solution by means of the
Lax-Milgram Lemma.

\par
Let $V^N\subset V$ be the finite element space of piecewise linear elements on the Shishkin mesh $\mathcal{T_N}$.
The SDFEM consists in adding weighted residuals to the standard Galerkin method in order to
stabilize the discretization. It reads:
\begin{equation}\label{eq:SDFEM}
\left\{
\begin{array}{lr}
\text{Find $u^{N}\in V^{N}$ such that for all $v^{N}\in V^{N}$},\\
 a_{SD}(u^{N},v^{N})=(f,v^{N})+\underset{K\subset\Omega}\sum(f,\delta_{K}b v^{N}_{x})_{K},
\end{array}
\right.
\end{equation}
where
\begin{equation*}
a_{SD}(u^{N},v^{N})=a_{Gal}(u^{N},v^{N})+a_{stab}(u^{N},v^{N})
\end{equation*}
and
\begin{align*}
a_{Gal}(u^{N},v^{N})&=\varepsilon (\nabla u^{N},\nabla v^{N})+(bu^{N}_{x}+cu^{N},v^{N}),\\
a_{stab}(u^{N},v^{N})&=\sum_{K\subset\Omega}(-\varepsilon\Delta u^{N}+bu^{N}_{x}+cu^{N},\delta_{K}bv^{N}_{x})_{K}.
\end{align*}
Note that $\Delta u^N=0$ in $K$ for $u^N\vert_K\in P_1(K)$ and $\delta_K=\delta(x,y)|_{K}$.  In this article, the stabilization parameter $\delta$ is chosen to be constant on each subdomain of  $\Omega$. Denote by $\delta_s$ the restriction of $\delta$ in $\Omega_s$ and similar $\delta_x$, $\delta_y$ and $\delta_{xy}$.


The SDFEM satisfies the following orthogonality 
\begin{equation}\label{eq:orthogonality of SD}
a_{SD}(u-u^N,v^N)=0,\quad \forall v^N\in V^N.
\end{equation}
Moreover, as shown in \cite{Roo1Sty2Tob3:2008-Robust}, if the stabilization parameter satisfies 
\begin{equation}\label{eq: delta-K}
0\le \delta_{K}\le \frac{\mu_0}{2\Vert c \Vert^2_{L^{\infty}(K)}},
\end{equation}
the SDFEM is coercive with respect to
the streamline diffusion norm
\begin{equation}\label{eq:SD coercivity}
a_{SD}(v^{N},v^{N})\ge \frac{1}{2} \Vert v^{N} \Vert^2_{SD},
\quad \forall v^{N}\in V^{N}
\end{equation}
where
\begin{equation}\label{eq:SD norm}
\Vert v^{N} \Vert^{2}_{SD}:=
\Vert v^{N} \Vert^{2}_{\varepsilon}
+\sum_{K\subset\Omega} \delta_{K}\Vert 
bv^{N}_x\Vert^{2}_{K}
\end{equation}
and  $\Vert v^{N} \Vert^{2}_{\varepsilon}:=
\varepsilon \vert v^{N} \vert^{2}_{1}+ \mu_{0}\Vert v^{N} \Vert^{2}$.
Note that  existence and uniqueness of the solution to \eqref{eq:SDFEM} is guaranteed by the  coercivity \eqref{eq:SD coercivity}.

\section{Integral inequalities and interpolation errors}\label{sub-section: preliminary}
In this section we present  integral inequalities for the diffusion and convection parts in the bilinear form and some interpolation bounds  for our later analysis. For notation convenience,  we set 
$$\partial^{l}_x\partial^{m}_y v:=\frac{\partial^{l+m}v}{\partial x^l\partial y^m}.$$

The following lemma will be used to  obtain sharp estimates of the diffusion part in the bilinear form $a_{SD}(\cdot,\cdot)$.
\begin{lemma}\label{lem:superconvergence-diffusion-term}
 Assume that $w\in  C^3(\bar{\Omega})$  and $v^N\in V^N$.   Let $w^I$ be the standard nodal linear interpolation on $\mathcal{T}_N$ and $l$, $m$ be nonnegative integers. If  $h_{y,j-1}=h_{y,j}$, then 
we have
\begin{equation*}
\left\vert \int_{\mathcal{Q}_{i,j}}  (w-w^{I})_{x}v^{N}_{x}\mr{d}x\mr{d}y\right\vert
\le 
C\sum_{l+m=2}h^{l}_{x,i}h^{m}_{y,j} \Vert \partial^{l+1}_x\partial^m_y w \Vert_{L^{\infty}(\mathcal{Q}_{i,j})}
\Vert v^{N}_{x} \Vert_{L^{1}(\mathcal{Q}_{i,j})},
\end{equation*}
where $\mathcal{Q}_{i,j}:=K^{1}_{i,j}\cup K^{2}_{i,j-1}$. If $h_{x,i-1}=h_{x,i}$, then we have
\begin{equation*}
\left\vert \int_{\mathcal{S}_{i,j}}  (w-w^{I})_{y}v^{N}_{y}\mr{d}x\mr{d}y\right\vert
\le 
C\sum_{l+m=2}h^{l}_{x,i}h^{m}_{y,j}\Vert \partial^{l}_x\partial^{m+1}_y w \Vert_{L^{\infty}(\mathcal{S}_{i,j})}
\Vert v^{N}_{y} \Vert_{L^{1}(\mathcal{S}_{i,j})}
\end{equation*}
where $\mathcal{S}_{i,j}:=K^{2}_{i-1,j}\cup K^{1}_{i,j}$.
\end{lemma}
\begin{proof}
Note that $v^N_x$ is a constant on the set $\mathcal{Q}_{i,j}$. First we expand $(w-w^{I})_{x}$  by Taylor's formula at $(x_i,y_j)$ with Lagrange form of the remainder.   After integration on $\mathcal{Q}_{i,j}$, we can offset  terms involving low derivatives of $w$.
Then  the first inequality is obtained.
The second inequality can be proved  similarly .  
See \cite[Lemma 2.1]{Zhan1Liu2:2015-Supercloseness-EL} for more details. 
\end{proof}

The following integral inequalities provide sharp estimates of  the convection part in the bilinear form $a_{SD}(\cdot,\cdot)$.
\begin{lemma}\label{lem:superconvergence-convection-term}
 Assume that $w\in C^{3}(\bar{\Omega})$ and let $w^I$ be the  piecewise linear
interpolation of $w$ on $\mathcal{T}_N$. Set
$\alpha=1$ or $\alpha=2$ and $p, q, l, m$ are nonnegative integers satisfying $0\le p+q\le 1$.
Suppose $h_{x,i-1}=h_{x,i}$, then we have
\begin{align}
&\left\vert \int_{K^{\alpha}_{i-1,j}}
 \partial^{p}_{x}\partial^{q}_{y}
(w-w^{I})\mr{d}x\mr{d}y-
\int_{K^{\alpha}_{i,j}}
 \partial^{p}_{x}\partial^{q}_{y}
(w-w^{I})\mr{d}x\mr{d}y\right\vert
\label{eq: w-wI-derivative-convection-x}\\
\le  &
C\sum_{l+m=3}h^{l+1-p}_{x,i}h^{m+1-q}_{y,j}\left \Vert
 \partial^{l}_{x}\partial^{m}_{y} w
  \right\Vert_{L^{\infty}(K^{\alpha}_{i-1,j}\cup K^{\alpha}_{i,j})}.
\nonumber
\end{align}
Suppose $h_{y,j-1}=h_{y,j}$, then we have
\begin{align}
  &\left\vert \int_{K^{\alpha}_{i,j-1}}
\partial^{p}_{x}\partial^{q}_{y}
 (w-w^{I})\mr{d}x\mr{d}y-
\int_{K^{\alpha}_{i,j}}
 \partial^{p}_{x}\partial^{q}_{y}
 (w-w^{I})\mr{d}x\mr{d}y\right\vert
\label{eq: w-wI-derivative-convection-y}\\
\le  &
C\sum_{l+m=3}h^{l+1-p}_{x,i}h^{m+1-q}_{y,j}\left \Vert
 \partial^{l}_{x}\partial^{m}_{y}
 w  \right\Vert_{L^{\infty}(K^{\alpha}_{i,j-1}\cup K^{\alpha}_{i,j})}.
\nonumber
\end{align}
\end{lemma}
\begin{proof}
We just prove \eqref{eq: w-wI-derivative-convection-x} for $\alpha=1$ and $p=q=0$.  The other estimates  can be  obtained similarly.

Expanding $(w-w^{I})_{x}$  by Taylor's formula at $(x_i,y_j)$,  we have
\begin{align*}
&(w-w^I)\vert_{K^1_{i,j}}=w(x,y)-(w(x_i,y_j)\lambda_1+w(x_{i+1},y_j)\lambda_2+
w(x_{i},y_{j+1})\lambda_3)\\
&=
w(x_i,y_j)+
\left( w_x(x_i,y_j)(x-x_i)+w_y(x_i,y_j)(y-y_j) \right)\\
&+\left(w_{xx}(x_i,y_j)\frac{(x-x_i)^2}{2}+w_{xy}(x_i,y_j)(x-x_i)(y-y_j) +w_{yy}(x_i,y_j)\frac{(y-y_j)^2}{2} \right)\\
&-\left(w(x_i,y_j)+w_x(x_i,y_j)h_{x,i}\lambda_2+w_{y}(x_i,y_j)h_{y,j}\lambda_3\right)\\
&-\left(w_{xx}(x_i,y_j)\frac{ h^2_{x,i} } {2}\lambda_2+w_{yy}(x_i,y_j)\frac{ h^2_{y,j} } {2}\lambda_3\right)+\mathcal{R}_{i,j},
\end{align*}
where $\lambda_1=1-\lambda_2-\lambda_3$,  $\lambda_2=\frac{x-x_i}{ h_{x,i} }$, $\lambda_3=\frac{ y-y_j }{ h_{y,j} }$ are the area basis functions and
\begin{equation}
\Vert \mathcal{R}_{i,j} \Vert_{L^{\infty}(K^{1}_{i,j})}\le
C\sum_{l+m=3}h^{l}_{x,i}h^{m}_{y,j}\left \Vert
\partial^{l}_{x}\partial^{m}_{y}
w \right\Vert_{L^{\infty}(K^{1}_{i,j})}\label{eq:super-convection-I-1}.
\end{equation}
Direct calculations yield
\begin{align}
\int_{ K^1_{i,j} }(w-w^I)\mr{d}x\mr{d}y
=&w_{xx}(x_i,y_j)\left(-\frac{h^{3}_{x,i}h_{y,j}}{24}\right)+w_{xy}(x_i,y_j)\left(\frac{h^{2}_{x,i}h^{2}_{y,j}}{24}\right)
\label{eq:super-convection-I-2}\\
&+w_{yy}(x_i,y_j)\left(-\frac{h_{x,i}h^{3}_{y,j} }{24}\right)
+\int_{ K^1_{i,j} } \mathcal{R}_{i,j} \mr{d}x\mr{d}y
\nonumber.
\end{align}
Similarly,  we have
\begin{align}
\int_{ K^1_{i-1,j} }(w-w^I)\mr{d}x\mr{d}y
=&w_{xx}(x_i,y_j)\left(-\frac{h^{3}_{x,i}h_{y,j}}{24}\right)
+w_{xy}(x_i,y_j)\left(\frac{h^{2}_{x,i}h^{2}_{y,j}}{24}\right)\label{eq:super-convection-II-1}\\
&+w_{yy}(x_i,y_j)\left(-\frac{h_{x,i}h^{3}_{y,j} }{24}\right)
+\int_{ K^1_{i-1,j} } \mathcal{R}_{i-1,j} \mr{d}x\mr{d}y
\nonumber
\end{align}
and
\begin{equation}
\Vert \mathcal{R}_{i-1,j} \Vert_{L^{\infty}(K^{1}_{i-1,j})}\le
\sum_{l+m=3}h^{l}_{x,i}h^{m}_{y,j}\left \Vert
\partial^{l}_{x}\partial^{m}_{y}
w \right\Vert_{L^{\infty}(K^{1}_{i-1,j})}\label{eq:super-convection-II-2}
\end{equation}
where the condition $h_{x,i-1}=h_{x,i}$ has been used in \eqref{eq:super-convection-II-1}. 

Combining \eqref{eq:super-convection-I-1}---\eqref{eq:super-convection-II-2}, we obtain
\eqref{eq: w-wI-derivative-convection-x} for $\alpha=1$ and $p=q=0$.
\end{proof}

For analysis on Shishkin meshes, we need the following  anisotropic interpolation error bounds given in
\cite[Lemma 3.2]{Guo1Styn2:1997-Pointwise}.
\begin{lemma}\label{lem: anisotropic interpolation-triangle}
Let $K\in\mathcal{T}_N$ and $p\in (1,\infty]$ and
suppose that $K$ is $K^1_{i,j}$ or $K^2_{i,j}$. Assume
that $w\in W^{2,p}(\Omega)$ and denote by $w^I$ the linear function that interpolates to $w$ at the vertices of
$K$.  Then
\begin{align*}
&\Vert w-w^I \Vert_{L^p(K)}\le
C\sum_{l+m=2}h^l_{x,i}h^m_{y,j}\Vert \partial^l_x\partial^m_y w \Vert_{L^p(K)},\\
&\Vert (w-w^I)_x \Vert_{L^p(K)}\le
C\sum_{l+m=1}h^l_{x,i}h^m_{y,j}\Vert \partial^{l+1}_x\partial^m_y w \Vert_{L^p(K)},\\
&\Vert (w-w^I)_y \Vert_{L^p(K)}\le
C\sum_{l+m=1}h^l_{x,i}h^m_{y,j}\Vert \partial^{l}_x\partial^{m+1}_y w \Vert_{L^p(K)},
\end{align*}
where $l$ and $m$ are nonnegative integers.
\end{lemma}
The following local estimates will also be frequently used.
\begin{lemma}\label{lem:linear interpolation}
Let $u^{I}$ and $E^{I}$ denote the piecewise linear
interpolation of $u$ and $E$, respectively,  on the Shishkin mesh $\mathcal{T}_{N}$, where $E$ can be any one of $E_{1}$, $E_{2}$ or $E_{12}$. Suppose that $u$ satisfies Assumption \ref{assumption-regularity}, then
\begin{align*}
&\Vert u-u^{I} \Vert_{L^{\infty}(K)}\le
\left\{
\begin{array}{ll}
CN^{-2},&\text{if $K \subset\Omega_{s}$}\\
CN^{-2}\ln^{2}N, &\text{otherwise}
\end{array}
\right. ,\\
&\Vert E^I \Vert_{L^{\infty}(\Omega_{s})}+\Vert \nabla E^I \Vert_{L^{1}(\Omega_{s})}\le CN^{-\rho},\\
&\Vert E^I_1 \Vert_{L^{\infty}(\Omega_{y})}+\Vert E^I_{12} \Vert_{L^{\infty}(\Omega_{y})}\le CN^{-\rho},\\
&\Vert ( E^I_{12})_y \Vert_{L^{1}(\Omega_{y})}\le C  N^{-(1+\rho)},\\
&\Vert \nabla E^I_1 \Vert_{L^{1}(\Omega_{y})}
+\Vert ( E^I_{12})_x \Vert_{L^{1}(\Omega_{y})}\le C\varepsilon^{1/2} N^{-\rho} \ln N.
\end{align*}
\end{lemma}
\begin{proof}
The first inequality  can be obtained in a similar way as \cite[Theorem 4.2]{Styn1ORior2:1997-uniformly}.
Here we only prove the second inequality for $E=E_1$ and the others can be proved similarly.
Recalling $E^I_1$ is the piecewise linear
interpolation of $E_1$, we have
\begin{equation}\label{eq:EI1-max}
\Vert E^I_1 \Vert_{L^{\infty}(\Omega_{s})}\le \Vert E_1 \Vert_{L^{\infty}(\Omega_{s})}\le CN^{-\rho}
\end{equation}
and 
\begin{align*}
  (E^I_1)_y|_{K^1_{i,j} }&=\frac{E_1(x_{i},y_{j+1})-E_1(x_i,y_{j})  }{ h_{y,j} } 
=(E_1)_y(x_i,\eta_j)
\end{align*}
 where $\eta_j\in (y_j,y_{j+1})$. Similarly, we have $ (E^I_1)_y|_{K^2_{i,j} }=(E_1)_y(x_{i+1},\tilde{\eta}_j)$ with  $\tilde{\eta}_j\in (y_j,y_{j+1})$ and
\begin{equation}\label{eq:EI1}
  (E^I_1)_x|_{K^1_{i,j} }=(E_1)_x(\xi_i,y_j),\quad
 (E^I_1)_x|_{K^2_{i,j} }=(E_1)_x(\tilde{\xi}_i,y_{j+1})
\end{equation}
where $\xi_i,\tilde{\xi}_i\in(x_i,x_{i+1})$. 
Recalling Assumption \ref{assumption-regularity}, we obtain 
\begin{align}
\left | (E^I_1)_y|_K \right|
\le
 C\Vert  (E_1)_y \Vert_{L^{\infty}(\Omega_s) }
  \le C N^{-\rho}, \; \forall K \subset \Omega_s. \nonumber
 \end{align}
Then we have
\begin{equation}\label{eq:EI1y-L1}
\Vert (E^I_1)_y \Vert_{L^{1}(\Omega_{s})}
\le 
C N^{-\rho}.
\end{equation}
Setting $\Omega_{s,r}:=\cup_{j=N/3}^{2N/3-2}\cup_{m=1}^2 K^m_{N/2-1,j} $.
 Recalling \eqref{eq:EI1} and Assumption \ref{assumption-regularity}, we obtain
\begin{align}
\Vert (E^I_1)_x \Vert_{L^1(\Omega_s\setminus\Omega_{s,r} )}
&=\sum_{i=0}^{N/2-2}\sum_{j=N/3}^{2N/3-1}\sum_{m=1}^2 \Vert (E^I_1)_x \Vert_{L^1(K^m_{i,j}) }\label{eq:EI1x-L1-I}\\
&\le 
C\sum_{i=0}^{N/2-2}  N^{-1}\varepsilon^{-1}e^{-\beta(1-x_{i+1})/\varepsilon}\nonumber \\
&\le
C\sum_{i=0}^{N/2-2}\int_{x_i}^{ x_{i+1} } \varepsilon^{-1}e^{-\beta(1-x)/\varepsilon}\mr{d}x\nonumber \\
&\le
C\int_{x_0}^{ x_{N/2} }\varepsilon^{-1}e^{-\beta(1-x)/\varepsilon}\mr{d}x
\le 
CN^{-\rho}\nonumber.
\end{align}
Note that $\mr{meas}(\Omega_{s,r})\le CN^{-1}$, then we have
\begin{align}
\Vert (E^I_1)_x \Vert_{L^1(\Omega_{s,r}) }&\le C N\Vert  E^I_1  \Vert_{L^1(\Omega_{s,r}) }\label{eq:EI1x-L1-II}\\
&\le C N \Vert E^I_1 \Vert_{L^{\infty}(\Omega_{s,r}) }\mr{meas}(\Omega_{s,r})\nonumber\\
&\le 
CN^{-\rho}\nonumber
\end{align}
where  inverse estimates  \cite[Theorem 3.2.6]{Ciarlet:1978-finite} have been used.

Now collecting \eqref{eq:EI1-max}, \eqref{eq:EI1y-L1}, \eqref{eq:EI1x-L1-I} and \eqref{eq:EI1x-L1-II}, we prove the inequality $\Vert E^I_1 \Vert_{L^{\infty}(\Omega_{s})}+\Vert \nabla E^I_1 \Vert_{L^{1}(\Omega_{s})}\le CN^{-\rho}$.
\end{proof}

%
%

\section{Supercloseness property on triangular meshes}
In this section, we will estimate each term in 
$a_{SD}(u-u^I,v^N)$ to derive the bound  of $\Vert u^I-u^N\Vert_{SD}$ on the Shishkin triangular mesh $\mathcal{T}_N$.
\begin{lemma}\label{lem: uI-uN-I}
Let $u$ be the solution of \eqref{eq:model problem} that satisfies Assumption \ref{assumption-regularity}, and
$u^I\in V^N$ be the linear interpolation of $u$ on the Shishkin mesh. Then  for all $v^N\in V^N$,  we have
$$ \left| \varepsilon (\nabla (u-u^I), \nabla v^{N})\right|
\le
C  (\varepsilon^{1/4}N^{-3/2}+N^{-3/2} )\ln^{3/2} N\;
\Vert v^{N} \Vert_{SD}.
$$
\end{lemma}
\begin{proof}
Recalling  the decomposition \eqref{eq:(2.1a)}, we set  $E=E_{1}+E_{2}+E_{12}$. Then we have
\begin{equation*}
( \nabla (u-u^I), \nabla v^{N} )=\mr{I}+\mr{II}+\mr{III},
\end{equation*}
where
\begin{align*}
\mr{I}:=&\left(  \nabla (E_{2}-E^I_{2}), \nabla  v^{N}\right)_{\Omega_{y} }+\left( \nabla (E-E^I),  \nabla v^{N} \right)_{\Omega_{xy} }+\left(  \nabla (S-S^I), \nabla v^{N} \right)_{\Omega_s\cup\Omega_y},\\
\mr{II}:= &\left(  \nabla (E-E^I), \nabla v^{N} \right)_{\Omega_s}
+\left(  \nabla (E_1-E^I_1), \nabla v^{N} \right)_{\Omega_y}+\left(  \nabla (S-S^I), \nabla v^{N} \right)_{ \Omega_{x y} }\\
&
+\left(  \nabla (E_{12}-E^I_{12}), \nabla v^{N} \right)_{\Omega_y},\\
\mr{III}:=&
\left(  \nabla (S-S^I), \nabla v^{N} \right)_{\Omega_x}
+\left(  (E_{1}-E^I_{1})_{y}, v^{N}_{y}\right)_{\Omega_x} 
+\left(  \nabla (E_2-E^I_2), \nabla v^{N} \right)_{\Omega_x}\\
&+\left(  \nabla (E_{12}-E^I_{12}), \nabla v^{N} \right)_{\Omega_x}+\left(  (E_{1}-E^I_{1})_{x}, v^{N}_{x}\right)_{\Omega_{x} }\\
=:&\mr{III}_1+\ldots+\mr{III}_4+\mr{III}_5.
\end{align*}

The estimates of $\mr{I}$ depend on Lemma \ref{lem:superconvergence-diffusion-term}.  Here we just present the detailed analysis for $((E_2-E^I_2)_y, v^N_y)_{\Omega_{y}^{d} }$ where $\Omega_{y}^{d}:=[0,1-\lambda_x]\times [0,\lambda_y]$,  since the other terms can be analyzed in a similar way.  First, we have

\begin{align*}
((E_2-E^I_2)_y, v^N_y)_{\Omega_{y}^{d} }
&=\sum_{i=0}^{N/2-1}\sum_{j=0}^{N/3-1}
\sum_{m=1}^2 ((E_2-E^I_2)_y, v^N_y)_{ K^{m}_{i,j} }\\
&=\sum_{j=0}^{N/3-1} ((E_2-E^I_2)_y, v^N_y)_{K^{1}_{0,j} }+\sum_{j=0}^{N/3-1} ((E_2-E^I_2)_y, v^N_y)_{K^{2}_{N/2-1,j} }\\
&\;\;\;+\sum_{i=1}^{N/2-1}\sum_{j=0}^{N/3-1}
((E_2-E^I_2)_y, v^N_y)_{ K^{2}_{i-1,j} \cup K^{1}_{i,j} }\\
&=:\mr{T}_1+\mr{T}_2+\mr{T}_3.
\end{align*}
Considering $v^N|_{\partial\Omega}=0$, we have
\begin{equation}\label{eq: diffusion-III-T1}
\mr{T}_1=0.
\end{equation}
H{\"o}lder inequalities and Lemma \ref{lem: anisotropic interpolation-triangle} yield 
\begin{align}
|\mr{T}_2|
&\le
 C \Vert  (E_2-E^I_2)_y\Vert_{L^{\infty}(\Omega_{y,r})}  \cdot \Vert v^N_y \Vert_{L^{1}(\Omega_{y,r})}\label{eq: diffusion-III-T2}\\
&\le
 C \varepsilon^{-1/2} N^{-1}\ln N\cdot \varepsilon^{1/4}N^{-1/2}\ln^{1/2}N\Vert v^N_y \Vert_{\Omega_y}
\nonumber\\
&\le
 C \varepsilon^{-3/4} N^{-3/2}\ln^{3/2} N\cdot \varepsilon^{1/2}\Vert v^N_y \Vert_{\Omega_y},
\nonumber
\end{align}
where $\Omega_{y,r}=\bigcup_{j=0}^{N/3-1}K_{N/2-1,j}^{2}$  and we have used
$\mr{meas}(\Omega_{y,r})\le C\varepsilon^{1/2} N^{-1}\ln N$.  
Using  Lemma \ref{lem:superconvergence-diffusion-term}, we obtain
\begin{align}
|\mr{T}_3|
&\le
 C \varepsilon^{-1/2} N^{-2}\ln^2 N\cdot \Vert v^N_y \Vert_{L^{1}(\Omega_{y}^{d})}\label{eq: diffusion-III-T3}\\
&\le
 C \varepsilon^{-1/2} N^{-2}\ln^2 N\cdot \varepsilon^{1/4}\ln^{1/2}N\Vert v^N_y \Vert_{\Omega_y}
\nonumber\\
&\le
 C \varepsilon^{-3/4} N^{-2}\ln^{5/2} N\cdot \varepsilon^{1/2}\Vert v^N_y \Vert_{\Omega_y}.
\nonumber
\end{align}
From \eqref{eq: diffusion-III-T1}---\eqref{eq: diffusion-III-T3}, we obtain
$$
|((E_2-E^I_2)_y, v^N_y)_{\Omega_y}|
\le 
 C \varepsilon^{-3/4} N^{-3/2}\ln^{3/2} N\;
\Vert v^{N} \Vert_{SD}.
$$
Similarly, we can estimate the remained terms in $\mr{I}$  and   obtain
\begin{equation}\label{eq: diffusion-I}
|\mr{I}|\le C \varepsilon^{-3/4}N^{-3/2}\ln^{3/2} N\;\Vert v^{N} \Vert_{SD}.
\end{equation}

Analysis of $\mr{II}$ depends on   Lemmas \ref{lem: anisotropic interpolation-triangle}, \ref{lem:linear interpolation} and/or smallness  of layer functions and layer domains.
For example, inverse estimates  \cite[Theorem 3.2.6]{Ciarlet:1978-finite} and Lemma \ref{lem:linear interpolation} yield
\begin{align*}
|\left(  \nabla (E-E^I), \nabla v^{N} \right)_{\Omega_s}|
\le &  \Vert  \nabla (E-E^I)  \Vert_{L^{1}(\Omega_{s})} \Vert \nabla v^{N} \Vert_{L^{\infty}(\Omega_{s})}
\\
\le & C N^{-\rho}\cdot N \Vert \nabla v^{N} \Vert_{\Omega_{s}}\\
\le & C\varepsilon ^{-1/2}N^{1-\rho} \Vert v^{N} \Vert_{SD}.
\end{align*}
Thus we obtain
\begin{equation}\label{eq: diffusion-II}
|\mr{II}|\le C(\varepsilon ^{-1/2}+ \varepsilon ^{-1/4}\ln^{1/2} N+\varepsilon ^{-3/4} N^{-1 }  \ln^{1/2} N)N^{1-\rho}\Vert v^{N} \Vert_{SD}.
\end{equation}

The analysis of $\mr{III}_1$--$\mr{III}_4$ is similar to one of $\mr{II}$ and the estimate of $\mr{III}_5$ is similar to one of $\mr{I}$. Thus we have
\begin{align}
&|\mr{III}_1|+\ldots+|\mr{III}_4|\le
C(N^{-1}\ln^{3/2} N+\varepsilon^{-1}N^{-\rho}\ln^{1/2} N)\Vert v^N \Vert_{SD},\\
&|\mr{III}_5|\le C \varepsilon^{-1}N^{-3/2}\ln^{3/2}N  \Vert v^N \Vert_{SD}.
\label{eq:III-5}
\end{align}

Collecting \eqref{eq: diffusion-I}--\eqref{eq:III-5},  the proof is done.
\end{proof}
\begin{lemma}\label{lem: uI-uN-II}
Let $u$ be the solution of \eqref{eq:model problem} that satisfies Assumption \ref{assumption-regularity}, and
$u^I\in V^N$ be the linear interpolation of $u$ on the Shishkin mesh.  Then  for all $v^N\in V^N$,  we have
\begin{equation}\label{eq:CR term in uI-uN}
\left| 
(b(u-u^{ I })_x+c(u-u^{I }), v^{N})\right|
\le C(A_s+A_y+(1+\varepsilon^{1/4}\ln N)
N^{-2}\ln^{5/2}N) \Vert v^N \Vert_{SD},
\end{equation}
where
\begin{align*}
A_s:=\min\{N^{-\rho}\delta_s^{-1/2}, N^{1-\rho} \},
A_y:=\varepsilon^{1/4}\min\{N^{-\rho}\delta_y^{-1/2}, N^{1-\rho} \} \ln^{1/2} N.
\end{align*}
\end{lemma}

\begin{proof}
Integration by parts  yields
\begin{align*}
(b(u-u^{ I })_x + c(u-u^{I }), v^{N})
&=-(b(u-u^I),v^N_x)+( (c-b_x)(u-u^{I }), v^{N} ).
\end{align*}
Lemma \ref{lem:linear interpolation}  yields
\begin{equation}\label{eq:C-0}
|( (c-b_x)(u-u^{I }), v^{N} )|\le C N^{-2}\ln^2N  \Vert v^N \Vert \le   C N^{-2}\ln^2N \Vert v^N \Vert_{SD}.
\end{equation} 
Recalling  the decomposition \eqref{eq:(2.1a)} and setting $E=E_{1}+E_{2}+E_{12}$, we have
$$
(b(u-u^I),v^N_x)=\mathcal{I}+\mathcal{II},
$$
where
\begin{align*}
\mathcal{I}:=&(b(E-E^I),v^N_x)_{\Omega_s}
+(b(u-u^I),v^N_x)_{\Omega_x\cup\Omega_{xy} }\\
&+(b(E_1+E_{12}-(E^I_1+E^I_{12}) ,v^N_x)_{\Omega_y},\\
\mathcal{II}
:=&(b(S-S^I),v^N_x)_{\Omega_s}
+(b(S-S^I),v^N_x)_{\Omega_y}+(b(E_2-E^I_2),v^N_x)_{\Omega_y}.
\end{align*} 

The analysis of $\mathcal{I}$ is similar to one of $\mr{II}$ in Lemma \ref{lem: uI-uN-I}. 
For example, 
\begin{align*}
|(b(E-E^I),v^N_x)_{\Omega_s}|
&\le
CN^{-\rho} \Vert v^N_x \Vert_{L^{1}(\Omega_s)}
\le
CN^{-\rho} \Vert v^N_x \Vert_{ \Omega_s }\\
&\le
\left\{
\begin{array}{l}
CN^{-\rho}\delta_s^{-1/2}\cdot \delta^{1/2}_s \Vert v^N_x \Vert_{ \Omega_s }\\
CN^{1-\rho}\cdot \Vert v^N \Vert_{ \Omega_s }
\end{array}
\right.
.
\end{align*}
%
Thus, we have
\begin{equation}\label{eq:C-I}
|\mathcal{I}|\le  C ( A_s +A_y+
N^{-2}\ln^{5/2}N) \Vert v^N \Vert_{SD},
\end{equation}
where
$
A_s:=\min\{N^{-\rho}\delta_s^{-1/2}, N^{1-\rho} \}$,
$A_y:=\varepsilon^{1/4}\min\{N^{-\rho}\delta_y^{-1/2}, N^{1-\rho} \} \ln^{1/2} N.
$
%

Next we are to analyze $\mathcal{II}$. Lemmas \ref{lem:superconvergence-convection-term}  and \ref{lem: anisotropic interpolation-triangle} yield
\begin{align}
&\left| 
(S-S^{I},w)_{ K^{m}_{i-1,j-(m-1)} }
-(S-S^{I},w)_{ K^{m}_{i,j-(m-1)} }
\right|
\label{eq:T3-part}\\
&\le
\left|(S-S^{I},w(x_i,y_j))_{ K^{m}_{i-1,j-(m-1)} }
-(S-S^{I},w(x_i,y_j))_{ K^{m}_{i,j-(m-1)} }\right|
\nonumber\\
&\;\;\;+\sum_{k=i-1}^{i}\left|(S-S^{I},w-w(x_i,y_j) )_{ K^{m}_{k,j-(m-1)} }\right|
\le
CN^{-5},
\nonumber
\end{align}
where $m=1$ or $2$,  $1\le i\le N/2-1$ and $N/3\le j-(m-1)\le 2N/3-1$. Also we have used  $w\in C^1(K)$, $\Vert w\Vert_{C^1(K)}\le C$  and  $\Vert w-w(x_i,y_j) \Vert_{L^{\infty}(K)}\le CN^{-1}$ where $K=K^{m}_{k,j-(m-1)}\subset\Omega_s$.
We decompose the first term of $\mathcal{II}$ as follows:
\begin{align*}
&(b(S-S^I),v^N_x)_{\Omega_s}
=\sum_{i=0}^{N/2-1}\sum_{j=N/3}^{2N/3-1}\sum_{m=1}^{2}
(S-S^{I},b v^{N}_{x})_{ K^{m}_{i,j} }\\
=&\frac{1}{H_x}\sum_{i=0}^{N/2-1}\sum_{j=N/3}^{2N/3-1}\sum_{m=1}^{2}
\left( S-S^{I},b (v^N(x_{i+1},y_{j+m-1})-v^N(x_i,y_{j+m-1}) ) \right)_{ K^{m}_{i,j} }
\\
=&-\frac{1}{H_x}\sum_{j=N/3}^{2N/3-1}\sum_{m=1}^{2}
(S-S^{I},b v^{N}(x_0,y_{j+m-1}) )_{K^{m}_{0,j}}\\
&+\frac{1}{H_x}\sum_{j=N/3}^{2N/3-1}\sum_{m=1}^{2}
(S-S^{I},b v^{N}(x_{N/2},y_{j+m-1}))_{K^{m}_{N/2-1,j}}
\\
&+\frac{1}{H_x}\sum_{j=N/3}^{2N/3-1}\sum_{i=1}^{N/2-1}\sum_{m=1}^{2}
v^{N}(x_{i},y_{j+m-1})\left(
(S-S^{I},b )_{K^{m}_{i-1,j}}
-(S-S^{I},b )_{K^{m}_{i,j}}
\right)
\\
=:&\mathscr{T}_1+\mathscr{T}_2+\mathscr{T}_3.
\end{align*}
Considering $w^{N}|_{\partial\Omega}=0$, we have
\begin{equation}\label{eq:T1}
\mathscr{T}_{1}=0.
\end{equation}
Note that $x_{N/2}=1-\lambda_{x}$, we have
\begin{align}
|\mathscr{T}_{2}|&\le
CN^{-3} \sum_{j=N/3}^{2N/3-1} |v^{N}(x_{N/2},y_{j})| 
\le
CN^{-3}\sum_{j=N/3}^{2N/3-1}
\left\vert \int_{ x_{N/2} }^{1} v^{N}_{x}(x,y_{j})\mr{d}x\right\vert
\label{eq:T2}\\
&\le
CN^{-3}\sum_{j=N/3}^{2N/3-1} \sum_{i=N/2}^{N-1}
\int_{ x_i }^{ x_{i+1} } \left\vert v^{N}_{x}(x,y_{j}) \right\vert \mr{d}x
\nonumber\\
&\le
CN^{-3}\cdot H^{-1}_{y}\sum_{j=N/3}^{2N/3-1}
\sum_{i=N/2}^{N-1}\sum_{m=1}^{2}\Vert v^{N}_{x} \Vert_{L^{1}(K^{m}_{i,j})}\le
CN^{-2}\Vert v^{N}_{x} \Vert_{L^{1}(\Omega_{x})}
\nonumber\\
&\le
CN^{-2}\cdot \varepsilon^{1/2}\ln^{1/2} N
\Vert v^{N}_{x} \Vert_{\Omega_{x}}
\le
CN^{-2}\ln^{1/2} N \Vert  v^{N} \Vert_{SD }.
\nonumber
\end{align}
 Using \eqref{eq:T3-part}, we obtain
\begin{align}
|\mathscr{T}_{3}|
&\le
C\frac{1}{H_{x}} \sum_{i=1}^{N/2-1} \sum_{j=N/3}^{2N/3-1} \sum_{m=1}^{2}
N^{-5}|v^{N}(x_{i},y_{j+m-1})|
\label{eq:T3}\\
&\le
CN^{-2} \Vert v^{N}\Vert_{L^{1}(\Omega_{s})}
\le
CN^{-2}\Vert v^{N}  \Vert_{SD}.
\nonumber
\end{align}
Collecting \eqref{eq:T1},\eqref{eq:T2} and \eqref{eq:T3}, we have
\begin{equation}
|(b(S-S^I),v^N_x)_{\Omega_s}|\le CN^{-2} \ln^{1/2} N \Vert v^{N}  \Vert_{SD}.
\end{equation}
Similarly, using  Lemma \ref{lem:superconvergence-convection-term} we have the estimates of the other terms of $\mathcal{II}$:
\begin{align*}
|(b(S-S^I),v^N_x)_{\Omega_y}|
&\le
C\varepsilon^{1/4}N^{-2}\ln N \Vert v^N \Vert_{SD}, \\
 |(b(E_2-E^I_2),v^N_x)_{\Omega_y}| 
&\le
C\varepsilon^{1/4}N^{-2}\ln^{7/2} N \Vert v^N \Vert_{SD}.
\end{align*}
Thus, we have
\begin{equation}\label{eq:C-II}
|\mathcal{II}|\le C N^{-2}(1+\varepsilon^{1/4}\ln^3 N)  \ln^{1/2} N\Vert v^N \Vert_{SD} .
\end{equation}

Collecting \eqref{eq:C-0}, \eqref{eq:C-I} and \eqref{eq:C-II}, the proof is done. 
\end{proof}

\begin{lemma}\label{lem: uI-uN-III}
Let Assumption \ref{assumption-regularity} hold true.
Suppose the stabilization parameter $\delta$ satisfies
\eqref{eq: delta-K}, then
\begin{align}
\left| 
a_{stab}(u-u^{I }, v^{N})
\right|
\le &
C (\delta_s\varepsilon\ln^{1/2}N+\delta^{1/2}_s  N^{-3/2} )   \Vert v^N\Vert_{SD}
\label{eq:stabilization term in uI-uN}\\
&+C\varepsilon^{1/4}(\delta_y+\delta^{1/2}_yN^{-3/2}\ln^{1/2} N)\ln N
\Vert v^N\Vert_{SD}
\nonumber\\
&+C\varepsilon^{-1}\delta_x\ln^{1/2}N
\Vert v^N\Vert_{SD}+C\varepsilon^{-3/4}\delta_{xy}\ln N
\Vert v^N\Vert_{SD}.
\nonumber
\end{align}

\end{lemma}
\begin{proof}
We have
$$
a_{stab}(u-u^{I },v^{N})=(-\varepsilon\Delta u+b(u-u^{I })_{x}+c(u-u^{I }),\delta bv^{N}_{x}).
$$

For $(\varepsilon\Delta u, \delta bv^{N}_{x})$, the reader is  referred to \cite[Theorem 5]{Fra1Lin2Roo3:2008-Superconvergence}. Its bound is
\begin{align}
|(\varepsilon\Delta u, \delta bv^{N}_{x})|\le&   C(\varepsilon  \delta_s \ln^{1/2}N+\delta^{1/2}_s N^{-3/2})\Vert v^N \Vert_{SD}\label{eq:stab-I}\\
&+C\varepsilon^{1/4} (  \delta_y \ln N+\delta_y^{1/2}N^{-3/2}\ln^{1/2} N ) \Vert v^N \Vert_{SD}\nonumber\\
&+C(\varepsilon^{-1}\delta_x\ln^{1/2}N+\varepsilon^{-3/4} \delta_{xy}\ln N ) \Vert v^N \Vert_{SD}.\nonumber
\end{align}

We can analyze $(b(u-u^I)_x,\delta bv^N_x)$ in a similar way as in Lemma  \ref{lem: uI-uN-I} and 
deal with $b$ as in \eqref{eq:T3-part}. Then we have
\begin{align}
|(b(u-u^I)_x,\delta bv^N_x)|
\le &
C(\delta_s^{1/2}  N^{-3/2}+ \delta_y^{1/2}\varepsilon^{1/4}N^{-3/2}\ln^{3/2}N)\Vert v^N \Vert_{SD}\label{eq:stab-II}\\
&+
C(\delta_x \varepsilon^{-1}N^{-1}\ln^{3/2}N+ \delta_{xy}  \varepsilon^{-3/4}N^{-1} \ln^{2}N)
\Vert v^N \Vert_{SD}.\nonumber
\end{align}

 According to the bounds of $(b(u-u^I),v^N_x)$ in Lemma \ref{lem: uI-uN-II}, we obtain
\begin{align}
|(c(u-u^{I }),\delta bv^{N}_{x})| \le 
&C(\delta_s^{1/2} N^{-\rho}+\delta_s N^{-2} \ln^{1/2} N)\Vert v^N \Vert_{SD}\label{eq:stab-III}\\
& +\varepsilon^{1/4} (\delta_y^{1/2}N^{-\rho}\ln^{1/2} N+\delta_y N^{-2}\ln^{7/2} N)\Vert v^N \Vert_{SD}\nonumber\\
&+ (\delta_x+\delta_{xy} ) N^{-2}\ln^{5/2}N \Vert v^N \Vert_{SD}.\nonumber
\end{align}

Collecting \eqref{eq:stab-I},  \eqref{eq:stab-II} and  \eqref{eq:stab-III}, we are done.
\end{proof}

\begin{theorem}\label{eq: uI-uN varepsilon}
Let Assumption \ref{assumption-regularity} hold true. Suppose  the stabilization parameter $\delta$ satisfies
\eqref{eq: delta-K} and
\begin{equation}\label{eq:delta}
 \delta_s \le C^*N^{-1/2},
\quad
\delta_y\le C^* \varepsilon^{-1/4}N^{-3/2},
\quad
 \delta_x\le C^*\varepsilon N^{-3/2},
\quad
 \delta_{xy}\le C^* \varepsilon^{3/4}N^{-3/2},
\end{equation}
where  $C^*$ is a positive constant independent of 
$\varepsilon$ and the mesh. Then we have
$$
\Vert u^I-u^{N} \Vert_{\varepsilon}
\le
\Vert u^I-u^{N} \Vert_{SD}
\le
C N^{-3/2}\ln^{3/2}N. $$
\end{theorem}
\begin{proof}
Considering  the coercivity \eqref{eq:SD coercivity} and orthogonality \eqref{eq:orthogonality of SD} of $a_{SD}(\cdot, \cdot)$, we have
\begin{align*}
\frac{1}{2}\Vert  u^{I}-u^{N} \Vert^{2}_{\varepsilon} 
&\le
\frac{1}{2}\Vert  u^{I}-u^{N} \Vert^{2}_{SD}
\le
a_{SD}(u^{I}-u, u^{I}-u^{N}).
\end{align*}
Taking $v^N=u^I-u^N$ in Lemmas \ref{lem: uI-uN-I}, \ref{lem: uI-uN-II} and \ref{lem: uI-uN-III},  the proof is finished.
\end{proof}
\begin{remark}
 The  convergence order of $\Vert u^I-u^N \Vert_{\varepsilon}$ is only $3/2$, as also appears in the following numerical tests (see \S 6). Note that  this convergence order is different from one in the case of  rectangular meshes, which is almost $2$ (see \cite[Theorem 5]{Fra1Lin2Roo3:2008-Superconvergence} and  \cite[Theorem 4.5]{Styn1Tobi2:2003-SDFEM}). 
\end{remark}
\begin{remark}
Theorem \ref{eq: uI-uN varepsilon} allows the construction of a simple postprocessing  as in \cite[Section 5.2]{Styn1Tobi2:2003-SDFEM}. A
local postprocessing of $u^N$ will yield a piecewise quadratic solution
$Pu^N$ for which in general $\Vert u-Pu^N \Vert_{\varepsilon}\ll \Vert u-u^N \Vert_{\varepsilon}$.
\end{remark}

%

%
%
\section{Supercloseness property on hybrid meshes}\label{sec: postprocess}

In this section, we will study an interesting problem which has been discussed in \cite{Roos:2006-Superconvergence} and \cite{Fra1Lin2Roo3:2008-Superconvergence}:  \emph{Where the use of bilinears has to be strongly recommended so that the bound $\Vert u^I-u^N\Vert_{SD}$ is of almost order 2}?
Careful observations of the proofs of Lemmas \ref{lem: uI-uN-I}, \ref{lem: uI-uN-II} and \ref{lem: uI-uN-III},  we find
that  in the case of triangles, only the term $\mr{III}_5$ in Lemma \ref{lem: uI-uN-I} and the stabilization parameter $\delta$ limit the order of 
$\Vert u^I-u^N\Vert_{SD}$. 
\begin{theorem}\label{theorem:superconvergence-Hybrid mesh-CL}
Suppose that Assumption \ref{assumption-regularity} holds true. Take $\delta_s=C^*N^{-1}$,  $\delta_y \le C^*\max\{N^{-3/2}, \varepsilon^{-1/4}N^{-2}\}$ and $\delta_x=\delta_{xy}=0$
where $C^*$ is a positive constant independent of $\varepsilon$ and the mesh such that $\delta$ satisfies \eqref{eq: delta-K}. For problems \eqref{eq:model problem}, if we use bilinear elements in $\Omega_x$ and linear elements in $\Omega\setminus\Omega_x$, we have
\begin{equation}
\Vert u^I-u^N \Vert_{SD}\le C(\varepsilon^{1/4}N^{-3/2}\ln^{3/2} N+N^{-2} \ln^2 N).
\end{equation}
\end{theorem}
\begin{proof}
Note that we use bilinear elements in $\Omega_x$ and linear elements in $\Omega\setminus\Omega_x$. Now we consider
$$
a_{SD}(u-u^I,v^N)=a_{SD; \Omega\setminus\Omega_x}(u-u^I,v^N)+a_{SD;\Omega_x}(u-u^I,v^N)
$$
where $a_{SD; \Omega_x}(\cdot,\cdot)$  and $a_{SD; \Omega\setminus\Omega_x}(\cdot,\cdot)$ mean the integrations in $a_{SD}(\cdot,\cdot)$ are restricted to $\Omega_x$ and $\Omega\setminus \Omega_x$ respectively.

According to Lemmas \ref{lem: uI-uN-I}, \ref{lem: uI-uN-II} and \ref{lem: uI-uN-III},
we have
\begin{align*}
|a_{SD; \Omega\setminus\Omega_x}(u-u^I,v^N)|
\le&
C\varepsilon^{1/4} N^{-3/2} \ln^{3/2} N\Vert v^N \Vert_{SD}\\
&+C(A_s+A_y+
(1+\varepsilon^{1/4}\ln^3 N)
N^{-2}\ln^{1/2}N)  ) \Vert v^N \Vert_{SD}\\
&+C (\delta_s\varepsilon\ln^{1/2}N+\delta^{1/2}_s  N^{-3/2} )   \Vert v^N\Vert_{SD}\\
&+C\varepsilon^{1/4}(\delta_y+\delta^{1/2}_yN^{-3/2} \ln^{1/2} N)\ln N \Vert v^N\Vert_{SD}\\
&+C\varepsilon^{-3/4}\delta_{xy}\ln N\Vert v^N\Vert_{SD},
\end{align*}
where $A_s$ and $A_y$ are defined as in Lemma \ref{lem: uI-uN-II}. 
Considering  the definitions of $\delta_s$,  $\delta_y$ and 
$\delta_{xy}$ and $\varepsilon \ln^6 N \le 1$,  we obtain
\begin{equation}\label{eq:hybrid-I}
|a_{SD; \Omega\setminus\Omega_x}(u-u^I,v^N)|
\le
C(\varepsilon^{1/4} N^{-3/2} \ln^{3/2} N+N^{-2}\ln N) \Vert v^N \Vert_{SD}.
\end{equation}

Note that $\delta_x=0$. According to \cite[Theorem 5]{Fran1Linb2:2007-Superconvergence}, we have
\begin{equation}\label{eq:hybrid-II}
|a_{SD; \Omega_x}(u-u^I,v^N)|
\le
CN^{-2}\ln^{2}N \Vert v^N \Vert_{SD}.
\end{equation}

%
Collecting \eqref{eq:hybrid-I} and \eqref{eq:hybrid-II},  we are done.
\end{proof}
\begin{remark}

Once we use linear elements in $\Omega_x$,  similar analysis   shows that $\Vert u^I-u^N \Vert_{SD}$ is of almost order $3/2$ again. Theorem \ref{theorem:superconvergence-Hybrid mesh-CL} shows that bilinear elements  should be recommended  for exponential layers to preserve $2$nd convergence of $\Vert u^I-u^N\Vert_{SD}$, and in the remained domain linear or bilinear elements could be used. 
\end{remark}
%
%

%
%
\section{Numerical results}
\noindent
In this section we give numerical results that appear to support our theoretical results. Errors and convergence rates of $u^I-u^N$ on Shishkin triangular meshes and hybrid meshes are presented. For the computations we set
\begin{equation*}
\delta_s = N^{-1},
\quad
\delta_y= N^{-3/2},
\quad
 \delta_x= \delta_{xy}=0.
\end{equation*}
All calculations were carried out by using Intel Visual Fortran 11. The discrete problems
were solved by the nonsymmetric iterative solver GMRES(c.f. e.g.,\cite{Ben1Gol2Lie3:2005-Numerical,Saad1Schu2:1986-GMRES}).

We will  illustrate our results  by computing errors and convergence orders for
the following boundary value problems
\begin{align*}
-\varepsilon\Delta u+(2-x)u_{x}+1.5u&=f(x,y)\quad&&\text{in $\Omega=(0,1)^{2}$},\\
u&=0&& \text{on $\partial\Omega$}
\end{align*}
where the right-hand side  $f$ is chosen such that
\begin{equation*}
u(x,y)=\left( \sin\frac{\pi x}{2}-\frac{e^{-(1-x)/\varepsilon }-e^{-1/\varepsilon }  } {1-e^{-1/\varepsilon } }\right)
\frac{ (1-e^{-y/\sqrt{\varepsilon} }) (1-e^{-(1-y)/\sqrt{\varepsilon}  })  } { 1-e^{ -1/\sqrt{\varepsilon} } } 
\end{equation*}
is the exact solution.

The errors in Tables \ref{table: triangle}--\ref{table: rectangle} are measured as follows
\begin{align*}
e^{N}_{SD}:&=\underset{\varepsilon=10^{-6},10^{-8},\ldots,10^{-16} } {\max} \left( \sum_{K\subset \Omega}\Vert u^I-u^N \Vert^{2}_{SD,K}\right)^{1/2},\\
e^{N}_{\varepsilon}:&=\underset{\varepsilon=10^{-6},10^{-8},\ldots,10^{-16} } {\max} \left( \sum_{K\subset \Omega}\Vert u^I-u^N \Vert^{2}_{\varepsilon,K}\right)^{1/2}.
\end{align*}
 The corresponding rates of convergence $p^{N}$ are computed from the formula
\begin{equation}\label{eq: convergence order formula}
    p^{N}=\frac{\ln e^{N}-\ln e^{2N}}{\ln2},
\end{equation}
where $e^{N}$ could be $e^{N}_{SD}$ or $e^{N}_{\varepsilon}$.

\begin{table}[htp]
\caption{Errors and convergence orders on Shishkin triangular meshes}
\footnotesize
\begin{tabular*}{\textwidth}{@{\extracolsep{\fill}}ccccc}
\hline
$N$   & $ \Vert u^I-u^{N}\Vert_{\varepsilon}$   & Rate   & $ \Vert u^I-u^{N}\Vert_{SD}$    & Rate\\
\hline
12   &$6.008\times10^{-2}$  & $1.14$    &$6.019\times10^{-2}$      &    $1.14$\\
24  &$2.727\times10^{-2}$   &$1.27$  &$2.729\times10^{-2}$      &    $1.27$\\
48  &$1.134\times10^{-2}$   &$1.33$   &$1.134\times10^{-2}$     &    $1.33$\\
96  &$4.511\times10^{-3}$   &$1.36$&$4.511\times10^{-3}$       &    $1.36$\\
192  &$1.757\times10^{-3}$   &$1.37$&$1.758\times10^{-3}$      &    $1.37$\\
384  &$6.788\times10^{-4}$   &$---$&$6.788\times10^{-4}$            &$---$\\
\hline
\end{tabular*}
\label{table: triangle}
\end{table}
In Table \ref{table: triangle}, the errors and convergence rates for $\Vert u^I-u^{N}\Vert_{\varepsilon}$ and $\Vert u^I-u^{N} \Vert_{SD}$ on the Shishkin triangular mesh are displayed.  We observe $\varepsilon$-independence of the errors and convergence rates. These numerical results support Theorem \ref{eq: uI-uN varepsilon}: almost $3/2$ order
 convergence for $\Vert u^I-u^{N}\Vert_{\varepsilon}$ and $\Vert u^I-u^{N} \Vert_{SD}$ on Shishkin triangular meshes. Also, Fig. \ref{fig:log-log chart} shows that the behavior of $\Vert u^I-u^{N} \Vert_{SD}$ is similar to $N^{-3/2}\ln^{3/4}N$ in the case of $\varepsilon=10^{-6},10^{-8},\cdots,10^{-16}$, as to some extent   supports Theorem \ref{eq: uI-uN varepsilon}.  

\begin{figure}[htp]
\centering
\includegraphics[width=0.6\textwidth]{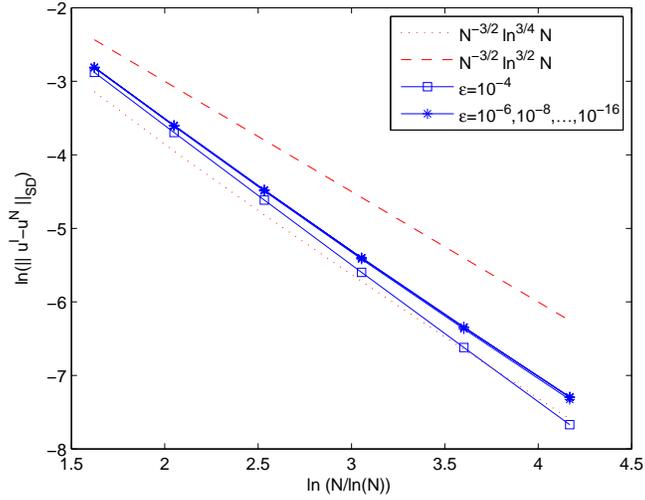}
\caption{Error $\Vert u^I-u^N \Vert_{SD}$ on the Shishkin triangular mesh.}
\label{fig:log-log chart}
\end{figure}

Tables \ref{table: hybrid I}, \ref{table: hybrid II}  and \ref{table: rectangle} present errors and convergence orders of  $\Vert u^I-u^{N}\Vert_{\varepsilon}$ and $\Vert u^I-u^{N} \Vert_{SD}$ on the hybrid mesh I, II and Shishkin rectangular mesh respectively. Among them, the hybrid mesh I consists of rectangles in $\Omega_x$ and triangles in $\Omega\setminus\Omega_x$, while the hybrid mesh II consists of triangles in $\Omega_x$ and rectangles in $\Omega\setminus\Omega_x$. Numerical results in Table \ref{table: hybrid I}  are similar with ones in Table \ref{table: rectangle} and support Theorem \ref{theorem:superconvergence-Hybrid mesh-CL}:
almost $2$ order
 convergence for $\Vert u^I-u^{N}\Vert_{\varepsilon}$ and $\Vert u^I-u^{N} \Vert_{SD}$. Besides, if we use linear elements in $\Omega_x$ and bilinear elements elsewhere, Table \ref{table: hybrid II} presents almost $3/2$  order convergence again and shows similarity  with Table \ref{table: triangle}. 

\begin{table}[htp]
\caption{Errors and convergence orders on the hybrid mesh I}
\footnotesize
\begin{tabular*}{\textwidth}{@{\extracolsep{\fill}}ccccc}
\hline
$N$   & $ \Vert u^I-u^{N}\Vert_{\varepsilon}$   & Rate   & $ \Vert u^I-u^{N}\Vert_{SD}$    & Rate\\
\hline
12   &$4.226\times10^{-2}$  & $1.24$    &$4.251\times10^{-2}$      &    $1.25$\\
24  &$1.786\times10^{-2}$   &$1.40$  &$1.789\times10^{-2}$      &    $1.40$\\
48  &$6.753\times10^{-3}$   &$1.51$   &$6.758\times10^{-3}$     &    $1.51$\\
96  &$2.368\times10^{-3}$   &$1.59$&$2.369\times10^{-3}$       &    $1.59$\\
192  &$7.886\times10^{-4}$   &$1.64$&$7.887\times10^{-4}$      &    $1.64$\\
384  &$2.531\times10^{-4}$   &$---$&$2.531\times10^{-4}$            &$---$\\
\hline
\end{tabular*}
\label{table: hybrid I}
\end{table}

\begin{table}[htp]
\caption{Errors and convergence orders on the hybrid mesh II}
\footnotesize
\begin{tabular*}{\textwidth}{@{\extracolsep{\fill}}ccccc}
\hline
$N$   & $ \Vert u^I-u^{N}\Vert_{\varepsilon}$   & Rate   & $ \Vert u^I-u^{N}\Vert_{SD}$    & Rate\\
\hline
12   &$6.122\times10^{-2}$  & $1.15$    &$6.136\times10^{-2}$      &    $1.16$\\
24  &$2.750\times10^{-2}$   &$1.27$  &$2.752\times10^{-2}$      &    $1.27$\\
48  &$1.139\times10^{-2}$   &$1.33$   &$1.139\times10^{-2}$     &    $1.33$\\
96  &$4.519\times10^{-3}$   &$1.36$&$4.519\times10^{-3}$       &    $1.36$\\
192  &$1.759\times10^{-3}$   &$1.37$&$1.759\times10^{-3}$      &    $1.37$\\
384  &$6.791\times10^{-4}$   &$---$&$6.791\times10^{-4}$            &$---$\\
\hline
\end{tabular*}
\label{table: hybrid II}
\end{table}

\begin{table}
\caption{Errors and convergence orders on Shishkin rectangular mesh}
\footnotesize
\begin{tabular*}{\textwidth}{@{\extracolsep{\fill}}ccccc}
\hline
$N$   & $ \Vert u^I-u^{N}\Vert_{\varepsilon}$   & Rate   & $ \Vert u^I-u^{N}\Vert_{SD}$    & Rate\\
\hline
12   &$4.230\times10^{-2}$  & $1.24$    &$4.242\times10^{-2}$      &    $1.25$\\
24  &$1.787\times10^{-2}$   &$1.41$  &$1.789\times10^{-2}$      &    $1.41$\\
48  &$6.745\times10^{-3}$   &$1.51$   &$6.749\times10^{-3}$     &    $1.51$\\
96  &$2.361\times10^{-3}$   &$1.59$&$2.362\times10^{-3}$       &    $1.59$\\
192  &$7.852\times10^{-4}$   &$1.64$&$7.854\times10^{-4}$      &    $1.64$\\
384  &$2.517\times10^{-4}$   &$---$&$2.517\times10^{-4}$            &$---$\\
\hline
\end{tabular*}
\label{table: rectangle}
\end{table}


\begin{thebibliography}{10}

\bibitem{Ben1Gol2Lie3:2005-Numerical}
M.~Benzi, G.H. Golub, and J.~Liesen.
\newblock Numerical solution of saddle point problems.
\newblock {\em Acta numerica}, 14(1):1--137, 2005.

\bibitem{Ciarlet:1978-finite}
P.G. Ciarlet.
\newblock {\em The finite element method for elliptic problems}.
\newblock North-Holland, 1978.

\bibitem{Fra1Kel2Sty3:2012-Galerkin}
S.~Franz, R.~B. Kellogg, and M.~Stynes.
\newblock {G}alerkin and streamline diffusion finite element methods on a
  {S}hishkin mesh for a convection-diffusion problem with corner singularities.
\newblock {\em Math. Comp.}, 81(278):661--685, 2012.

\bibitem{Fran1Linb2:2007-Superconvergence}
S.~Franz and T.~Lin{\ss}.
\newblock Superconvergence analysis of the {G}alerkin {FEM} for a singularly
  perturbed convection--diffusion problem with characteristic layers.
\newblock {\em Numer. Meth. Part. D. E.}, 24(1):144--164, 2007.

\bibitem{Fra1Lin2Roo3:2008-Superconvergence}
S.~Franz, T.~Lin{\ss}, and H.-G. Roos.
\newblock Superconvergence analysis of the {SDFEM} for elliptic problems with
  characteristic layers.
\newblock {\em Appl. Numer. Math.}, 58:1818--1829, 2008.

\bibitem{Guo1Styn2:1997-Pointwise}
W.~Guo and M.~Stynes.
\newblock Pointwise error estimates for a streamline diffusion scheme on a
  {S}hishkin mesh for a convection--diffusion problem.
\newblock {\em IMA J. Numer. Anal.}, 17:29--59, 1997.

\bibitem{Hugh1Broo2:1979-multidimensional}
T.J.R. Hughes and A.~Brooks.
\newblock A multidimensional upwind scheme with no crosswind diffusion.
\newblock In Thomas J.~R. Hughes, editor, {\em Finite element methods for
  convection dominated flows}, volume AMD 34. American Society of Mechanical
  Engineers. Applied Mechanics Division, New York, 1979.

\bibitem{Kell1Styn2:2005-Corner}
R.B. Kellogg and M.~Stynes.
\newblock Corner singularities and boundary layers in a simple
  convection-diffusion problem.
\newblock {\em J. Differential Equations}, 213:81--120, 2005.

\bibitem{Kell1Styn2:2007-Sharpened}
R.B. Kellogg and M.~Stynes.
\newblock Sharpened bounds for corner singularities and boundary layers in a
  simple convection-diffusion problem.
\newblock {\em Appl. Math. Lett.}, 20:539--544, 2007.

\bibitem{Lin1Yan2Zho3:1991-rectangle}
Q.~Lin, N.N. Yan, and A.H. Zhou.
\newblock A rectangle test for interpolated finite elements.
\newblock In {\em Proc. Sys. Sci. and Sys. Eng.(Hong Kong), Great Wall Culture
  Publ. Co}, pages 217--229, 1991.

\bibitem{Roos:2006-Superconvergence}
H.-G. Roos.
\newblock Superconvergence on a hybrid mesh for singularly perturbed problems
  with exponential layers.
\newblock {\em ZAMM Z. Angew. Math. Me.}, 86(8):649--655, 2006.

\bibitem{Roo1Sty2Tob3:2008-Robust}
H.~G. Roos, M.~Stynes, and L.~Tobiska.
\newblock {\em Robust numerical methods for singularly perturbed differential
  equations: convection-diffusion-reaction and flow problems}.
\newblock Springer-Verlag, Berlin Heidelberg, 2008.

\bibitem{Saad1Schu2:1986-GMRES}
Y.~Saad and M.H. Schultz.
\newblock {GMRES}: {A} generalized minimal residual algorithm for solving
  nonsymmetric linear systems.
\newblock {\em SIAM Journal on scientific and statistical computing},
  7(3):856--869, 1986.

\bibitem{Shishkin:1990-Grid}
G.I. Shishkin.
\newblock {\em Grid approximation of singularly perturbed elliptic and
  parabolic equations}.
\newblock Second doctorial thesis, Keldysh Institute, Moscow, 1990.
\newblock In Russian.

\bibitem{Stynes:2005-Steady}
M.~Stynes.
\newblock Steady-state convection-diffusion problems.
\newblock {\em Acta Numer.}, 14:445--508, 2005.

\bibitem{Styn1ORior2:1997-uniformly}
M.~Stynes and E.~O'Riordan.
\newblock A uniformly convergent {G}alerkin method on a {S}hishkin mesh for a
  convection-diffusion problem.
\newblock {\em J. Math. Anal. Applic.}, 214:36--54, 1997.

\bibitem{Styn1Tobi2:2003-SDFEM}
M.~Stynes and L.~Tobiska.
\newblock The {SDFEM} for a convection--diffusion problem with a boundary
  layer: optimal error analysis and enhancement of accuracy.
\newblock {\em SIAM J. Numer. Anal.}, 41(5):1620--1642, 2003.

\bibitem{Styn1Tobi2:2008-Using}
M.~Stynes and L.~Tobiska.
\newblock Using rectangular {$Q_p$} elements in the {SDFEM} for a
  convection--diffusion problem with a boundary layer.
\newblock {\em Appl. Numer. Math.}, 58(12):1789--1802, 2008.

\bibitem{Zhan1Liu2:2015-Supercloseness-EL}
J.~Zhang and X.W. Liu.
\newblock Supercloseness of the {SDFEM} on {S}hishkin triangular meshes for
  problems with exponential layers.
\newblock submitted, 2015.

\end{thebibliography}
\end{document}